\documentclass[a4paper,12pt]{amsart}

\usepackage{amsfonts}
\usepackage{amsmath}
\usepackage{amssymb}
\usepackage{mathrsfs}
\usepackage{hyperref}
\usepackage{graphicx}

\usepackage[usenames]{color}

\numberwithin{equation}{section}
\theoremstyle{plain}
\newtheorem{thm}{Theorem}[section]
\newtheorem{prop}[thm]{Proposition}
\newtheorem{cor}[thm]{Corollary}
\newtheorem{defn}[thm]{Definition}

\newtheorem{rem}[thm]{Remark}
\def\R{\mathbb{R}^{N}}
\def\G{\mathbb{G}}

\begin{document}

 \title[Stein-Weiss-Adams inequality on Morrey spaces]
{Stein-Weiss-Adams inequality on Morrey spaces}

\author[A. Kassymov]{Aidyn Kassymov}
\address{
  Aidyn Kassymov:
  \endgraf
   \endgraf
  Department of Mathematics: Analysis, Logic and Discrete Mathematics
  \endgraf
  Ghent University, Belgium
  \endgraf
	and
\endgraf
  Institute of Mathematics and Mathematical Modeling
  \endgraf
  125 Pushkin str.
  \endgraf
  050010 Almaty
  \endgraf
  Kazakhstan
  \endgraf
  {\it E-mail address} {\rm aidyn.kassymov@ugent.be} and {\rm kassymov@math.kz}}

\author[M. A. Ragusa]{Maria Alessandra Ragusa}
\address{
	Maria Alessandra Ragusa:
	 \endgraf
  Dipartimento di Matematica e Informatica, Universit\'{a} di Catania 
  \endgraf
  Catania, Italy
  \endgraf
  and
  \endgraf
  Faculty of Fundamental Science, Industrial University
    \endgraf
    Ho Chi Minh City, Vietnam
	\endgraf
  {\it E-mail address} {\rm maragusa@dmi.unict.it}
}
  \author[M. Ruzhansky]{Michael Ruzhansky}
\address{
	Michael Ruzhansky:
	 \endgraf
  Department of Mathematics: Analysis, Logic and Discrete Mathematics
  \endgraf
  Ghent University, Belgium
  \endgraf
  and
  \endgraf
  School of Mathematical Sciences
    \endgraf
    Queen Mary University of London
  \endgraf
  United Kingdom
	\endgraf
  {\it E-mail address} {\rm michael.ruzhansky@ugent.be}
}

\author[D. Suragan]{Durvudkhan Suragan}
\address{
	Durvudkhan Suragan:
	\endgraf
	Department of Mathematics, Nazarbayev University
	\endgraf
	53 Kabanbay Batyr Ave,  010000 Astana
	\endgraf
	Kazakhstan
	\endgraf
	{\it E-mail address} {\rm durvudkhan.suragan@nu.edu.kz}}


    \keywords{Riesz potential, Stein-Weiss inequality, fractional operator, global Morrey space,
     	 homogeneous Lie group.}
 \subjclass{22E30, 43A80.}
     \maketitle

     \begin{abstract}
We establish Adams type Stein-Weiss inequality on global Morrey spaces on general homogeneous groups.  Special properties of homogeneous norms and some boundedness results on global Morrey spaces play key roles in our proofs.  As consequence, we obtain fractional Hardy, Hardy-Sobolev, Rellich and Gagliardo-Nirenberg inequalities on Morrey spaces on stratified groups. While the results are obtained in the setting of general homogeneous groups, they are new already for the Euclidean space $\mathbb{R}^{N}.$
     \end{abstract}
\tableofcontents
\section{Introduction}

In their pioneering work \cite{HL28},  Hardy and  Littlewood considered the one dimensional fractional integral operator on $(0,\infty)$ given by
\begin{equation}\label{1Doper}
T_{\lambda}u(x)=\int_{0}^{\infty}\frac{u(y)}{|x-y|^{\lambda}}dy,\quad 0<\lambda<1,
\end{equation}
and proved the following inequality:
\begin{thm}\label{1DHLS28}
Let $1<p<q<\infty$ and $u\in L^{p}(0,\infty)$  with $\frac{1}{q}=\frac{1}{p}+\lambda-1$. Then
\begin{equation}
\|T_{\lambda}u\|_{L^{q}(0,\infty)}\leq C \|u\|_{L^{p}(0,\infty)},
\end{equation}
where $C$ is a positive constant independent of $u$.
\end{thm}
The $N$-dimensional analogue of \eqref{1Doper} can be written by the formula:
\begin{equation}\label{NDoper}
I_{\lambda}u(x)=\int_{\mathbb{R}^{N}}\frac{u(y)}{|x-y|^{\lambda}}dy,\,\,\,\,0<\lambda<N.
\end{equation}
The $N$-dimensional case of  Theorem \ref{1DHLS28} was obtained by Sobolev in
\cite{Sob38}:
\begin{thm}\label{THM:HLS}
Let $u\in L^{p}(\mathbb{R}^{N})$ and $1<p<q<\infty$ with $\frac{1}{q}=\frac{1}{p}+\frac{\lambda}{N}-1$. Then
\begin{equation}
\|I_{\lambda}u\|_{L^{q}(\mathbb{R}^{N})}\leq C \|u\|_{L^{p}(\mathbb{R}^{N})},
\end{equation}
where $C$ is a positive constant independent of $u$.
\end{thm}
Later, in \cite{StWe58}  Stein and  Weiss  obtained the following  two-weight extention of  the  Hardy-Little$\-$wood-Sobolev inequality, which is known as the Stein-Weiss inequality or (radial) weighted Hardy-Littlewood-Sobolev inequality:
\begin{thm}\label{Classiacal_Stein-Weiss_inequality}
Let $0<\lambda<N$, $1<p<\infty$,

\begin{equation}
\alpha<\frac{N(p-1)}{p}, \,\,\,\,\beta<\frac{N}{q},\,\,\,\,\alpha+\beta\geq0,
\end{equation}
and
\begin{equation}
\frac{1}{q}=\frac{1}{p}+\frac{\lambda+\alpha+\beta}{N}-1.
\end{equation}
If $1<p\leq q<\infty$, then
\begin{equation}
\||x|^{-\beta}I_{\lambda}u\|_{L^{q}(\mathbb{R}^{N})}\leq C \||x|^{\alpha}u\|_{L^{p}(\mathbb{R}^{N})},
\end{equation}
where $C$ is a positive constant independent of $u$.
\end{thm}

The Hardy-Littlewood-Sobolev inequality on Euclidean spaces and the regularity of fractional integrals were studied in \cite{CF}, \cite{FM}, \cite{MW}  and \cite{Per}. In \cite{FS74} Folland and Stein obtained the Hardy-Littlewood-Sobolev inequality
on the Heisenberg group. We also note that the best constant of the Hardy-Littlewood-Sobolev inequality for the Heisenberg group is now known, see Frank and Lieb \cite{FL12} (in the Euclidean case this was done earlier by Lieb in \cite{Lie83}). The expression for the best constant depends on the particular quasi-norm used and may change for a different choice of a quasi-norm.
In  \cite{GMS} the authors studied the Stein-Weiss inequality on the Carnot groups. On homogeneous Lie groups, the Hardy-Littlewood-Sobolev and Stein-Weiss inequalities were obtained in \cite{KRS},  \cite{KRS2} and \cite{RY}.

In the Euclidean setting, Spanne \cite{P69} and Adams \cite{A} investigated the boundedness of the Riesz potential operator on the global Morrey space (that is, the Morrey space Hardy-Littlewood-Sobolev inequality). Their results can be stated as follows:

\begin{thm}[Spanne's result published by Peetre in \cite{P69}] \label{Peetre}
Let $0<\gamma<N$, $1<p<\frac{N}{\gamma}$, $0<\lambda<N-\gamma p$, $\frac{1}{q}=\frac{1}{p}-\frac{\gamma}{N}$ and $\frac{\lambda}{p}=\frac{\mu}{q}$. Then
\begin{equation}
\|I_{N-\gamma}u\|_{M^{\mu}_{q}(\mathbb{R}^{N})}\leq C \|u\|_{M^{\lambda}_{p}(\mathbb{R}^{N})},
\end{equation}
where $C$ is a positive constant independent of $u$.
\end{thm}

\begin{thm}[Adams's result \cite{A}] \label{Adams}
Let $0<\gamma<N$, $1<p<\frac{N}{\gamma}$, $0<\lambda<N-\gamma p$, $\frac{1}{q}=\frac{1}{p}-\frac{\gamma}{N-\lambda}$. Then
\begin{equation}
\|I_{N-\gamma}u\|_{M^{\lambda}_{q}(\mathbb{R}^{N})}\leq C \|u\|_{M^{\lambda}_{p}(\mathbb{R}^{N})},
\end{equation}
where $C$ is a positive constant independent of $u$.
\end{thm}

 For the general global Morrey space, the Spanne type Hardy-Littlewood-Sobolev inequality on the Euclidean space was investigated in \cite{N94,PS10} and on general metric measure space in \cite{KM, S22}. The Spanne type Stein-Weiss inequality was also studied in \cite{Ho}.
Weighted estimates on global Morrey spaces in the Adams range were obtained in \cite{NST18, T11}. Note that the unweighted Adams type inequality in the local Morrey spaces does not hold (see \cite{KS17}).
Some complements
of both Theorem \ref{Peetre} and Theorem \ref{Adams} were considered in the recent paper \cite{NS22}.

The main goal of this paper is to obtain the Stein-Weiss extension of Theorem \ref{Adams}. We state our result on homogeneous Lie groups. In particular, the obtained result recovers the previously known results on Euclidean (Abelian), Heisenberg, and Carnot (stratified) groups since the class of the homogeneous Lie groups contains those. However, our main result Theorem \ref{stein-weissthm} seems new even in the Euclidean setting.
Note that in this direction systematic studies of different functional inequalities on general homogeneous (Lie) groups were initiated by the paper \cite{RSAM}. We refer to this and other papers by the authors (e.g. \cite{RSYm}) for further discussions.

Summarising our main results in this paper, we present the following facts:
\begin{itemize}
\item Hardy-Littlewood-Sobolev inequality on  the global Morrey space;
\item Stein-Weiss-Adams inequality on the global Morrey space;
\item Morrey space fractional (and integer) Hardy and Rellich inequalities;
\item Morrey space fractional (and integer) Hardy-Sobolev inequality;
\item Morrey space fractional (and integer) Gagliardo-Nirenberg inequality.
\end{itemize}

\section{Preliminaries}
\label{SEC:2}
\subsection{Homogeneous groups} A Lie group (on $\mathbb{R}^{N}$) $\mathbb{G}$
is said to be homogeneous if there is a dilation $D_{\lambda}(x)$ such that
$$D_{\lambda}(x):=(\lambda^{\nu_{1}}x_{1},\ldots,\lambda^{\nu_{N}}x_{N}),\; \nu_{1},\ldots, \nu_{n}>0,\; D_{\lambda}:\mathbb{R}^{N}\rightarrow\mathbb{R}^{N},$$
which is an automorphism of the group $\mathbb{G}$ for each $\lambda>0$. For simplicity, throughout this paper we use the notation $\lambda x$ for the dilation $D_{\lambda} (x).$  The homogeneous dimension of  $\mathbb{G}$ is denoted by $Q:=\nu_{1}+\ldots+\nu_{N}.$
In this paper, we denote a homogeneous quasi-norm on $\mathbb{G}$ by $|x|$, which
is a continuous non-negative function
\begin{equation}
\mathbb{G}\ni x\mapsto |x|\in[0,\infty),
\end{equation}
with the properties

\begin{itemize}
	\item[i)] $|x|=|x^{-1}|$ for all $x\in\mathbb{G}$,
	\item[ii)] $|\lambda x|=\lambda |x|$ for all $x\in \mathbb{G}$ and $\lambda>0$,
	\item[iii)] $|x|=0$ iff $x=0$.
\end{itemize}
Moreover, the following polarisation formula on homogeneous (Lie) groups will be used in our proofs:
there is a (unique)
positive Borel measure $\sigma$ on the
unit quasi-sphere
$
\mathfrak{S}:=\{x\in \mathbb{G}:\,|x|=1\},
$
so that for every $f\in L^{1}(\mathbb{G})$ we have (see \cite{FS1})
\begin{equation}\label{EQ:polar}
\int_{\mathbb{G}}f(x)dx=\int_{0}^{\infty}
\int_{\mathfrak{S}}f(ry)r^{Q-1}d\sigma(y)dr.
\end{equation}
The quasi-ball centred at $x \in \mathbb{G}$ with radius $R > 0$ can be defined by
\begin{equation*}
B(x,R) := \{y \in\mathbb {G} : |x^{-1} y|< R\}.
\end{equation*}
A homogeneous group is necessarily nilpotant and unimodular, and the
Haar measure on $\mathbb{G}$ coincides with the Lebesgue measure (see e.g. \cite[Proposition 1.6.6]{FR}); we will denote it by $dx$. If $|S|$ denotes the volume of a measurable set $S\subset \mathbb{G}$, then
\begin{equation}\label{scal}
|D_{\lambda}(S)|=\lambda^{Q}|S| \quad {\rm and}\quad \int_{\mathbb{G}}f(\lambda x)
dx=\lambda^{-Q}\int_{\mathbb{G}}f(x)dx.
\end{equation}
Hence, we have that the Haar measure of the quasi-ball has the following estimate
\begin{equation}\label{ballmeasure}
    C^{-1}R^{Q}\leq |B(x,R)|\leq CR^{Q}.
\end{equation}

For discussions on properties of the homogeneous group, we refer to books \cite{FS1}, \cite{FR} and \cite{RS_book}.

Let us consider the integral operator
\begin{equation}
I_{\gamma}u(x)=|x|^{\gamma-Q}*u(x)=\int_{\mathbb{G}}|y^{-1} x|^{\gamma-Q}u(y)dy,\,\,\,0<\gamma <Q,
\end{equation}
where $*$ is the convolution.

Let us introduce the Hardy-Littlewood maximal operator and the fractional Hardy-Littlewood maximal operator:
\begin{equation}
    M_{0}u(x):=\sup_{r>0}\frac{1}{|B(x,r)|}\int_{B(x,r)}|u(y)|dy,
\end{equation}
and
\begin{equation}
    M_{\alpha}u(x):=\frac{1}{|B(x,r)|^{1-\alpha}}\int_{B(x,r)}|u(y)|dy,
\end{equation}
respectively.

Let us also recall the following well-known fact about quasi-norms on homogeneous (Lie) groups.
\begin{prop}[\cite{FR}, Theorem 3.1.39] \label{prop_quasi_norm}
Let $\mathbb{G}$ be a homogeneous group. Then there exists a homogeneous
quasi-norm on $\mathbb{G}$ which is a norm, that is, a homogeneous quasi-norm $|\cdot|$
which satisfies the triangle inequality
\begin{equation}\label{tri}
|x y|\leq |x| + |y|, \,\,\,\forall x, y \in \mathbb{G}.
\end{equation}
Furthermore, all homogeneous quasi-norms on $\mathbb{G}$ are equivalent.
\end{prop}
\subsection{Stratified groups}
In this subsection, we recall the definition of the homogeneous stratified group (or homogeneous Carnot group). 
It is an important class of homogeneous groups. In this environment, the theory of basic function inequalities
becomes intricately intertwined with the properties of sub-Laplacians. We refer  \cite{BLU07}, \cite{FR} and \cite{RS_book} for further discussions in this direction. 
\begin{defn} \label{maindef}
A Lie group $\mathbb{G}=(\mathbb{R}^{N},\circ)$ is called a homogeneous stratified group if it satisfies the following assumptions:

(a) For some natural numbers $N_{1}+...+N_{r}=N$
the decomposition $\mathbb{R}^{N}=\mathbb{R}^{N_{1}}\times...\times\mathbb{R}^{N_{r}}$ is valid, and
for every $\lambda>0$ the dilation $\delta_{\lambda}: \mathbb{R}^{N}\rightarrow \mathbb{R}^{N}$
given by
$$\delta_{\lambda}(x)\equiv\delta_{\lambda}(x^{(1)},...,x^{(r)}):=(\lambda x^{(1)},...,\lambda^{r}x^{(r)})$$
is an automorphism of the group $\mathbb{G}.$ Here $x^{(k)}\in \mathbb{R}^{N_{k}}$ for $k=1,...,r.$

(b) Let $N_{1}$ be as in (a) and let $X_{1},...,X_{N_{1}}$ be the left-invariant vector fields on $\mathbb{G}$ such that
$X_{k}(0)=\frac{\partial}{\partial x_{k}}|_{0}$ for $k=1,...,N_{1}.$ Then
$${\rm rank}({\rm Lie}\{X_{1},...,X_{N_{1}}\})=N,$$
for every $x\in\mathbb{R}^{N},$ i.e. the iterated commutators
of $X_{1},...,X_{N_{1}}$ span the Lie algebra of $\mathbb{G}.$
\end{defn}

The notation
$$\nabla_{\mathbb{G}}:=(X_{1},\ldots, X_{N_{1}})$$
presents the horizontal gradient on $\G$.
So, the sub-Laplacian on (homogeneous) stratified groups is given by
$$\Delta_{\G}:=\nabla_{\G} \cdot \nabla_{\G}.$$

Note that a Lie group is called stratified if it is connected and simply-connected
Lie group whose Lie algebra is stratified. Any (abstract) stratified group is isomorphic to a homogeneous one. 
In this paper, we will refer to the homogeneous stratified (Lie) group as the "stratified group" for the sake of simplicity and clarity.

\subsection{Morrey spaces}
%
Let $\G$ be a homogeneous group, $1 < p < +\infty$, $0 < \lambda < Q$, a real function
 $f\in L^{p}_{\text{loc}}(\G)$ and the set described by the quantity
\begin{equation}
    r^{-\lambda}\int_{B(x,r)}|f(y)|^{p}dy, x\in\G,
\end{equation}
is upper bounded, then
we say that $f$ belongs to the  global Morrey space $M^{\lambda}_{p}(\G)$;  this space can be endowed with the norm
\begin{equation}
    \|f\|^{p}_{M^{\lambda}_{p}(\G)}:=\sup\limits_{x\in \G}\sup\limits_{r>0}r^{-\lambda}\int_{ B(x,r)}|f(y)|^{p}dy.
\end{equation}
Similarly, with the previous definition, we introduce the  local Morrey space ${LM^{\lambda}_{p}(\G)}$ with the norm
\begin{equation}
    \|f\|^{p}_{LM^{\lambda}_{p}(\G)}:=\sup\limits_{r>0}r^{-\lambda}\int_{B(e,r)}|f(y)|^{p}dy,
\end{equation}
where $e$ is the identity element of $\G$.

Also, it is well known  if $\lambda=0$ then $M^{0}_{p}(\G)=L^{p}(\G)$ and if $\lambda=Q$ we have $M^{Q}_{p}(\G)=L^{\infty}(\G)$. For more details on this topic, we refer to \cite{ADbook},  \cite{Rag08} and \cite{Rag12} .

\subsection{Fractional Laplacians}

Here, we briefly define the fractional Laplacian on $\mathbb{R}^{N}$. For the convenience of readers, we give the main references to the fractional Laplacian in \cite{DVP12,G22}. By $\mathcal{S}$ we denote  the Schwartz space  of rapidly
decaying $C^{\infty}$ functions in $\mathbb{R}^{N}$. We define the fractional Laplacian $(-\Delta)^{s}$  as
\begin{defn}
Let $0 < s < 1$. The fractional Laplacian of a function is the nonlocal operator in $\mathbb{R}^{N}$ defined by the expression
    \begin{equation}\label{fraclap}
    (-\Delta)^{s}u(x):=\frac{A(N,s)}{2}\int_{\mathbb{R}^{N}}\frac{2u(x) - u(x + y) - u(x -y)}{|y|^{N+2s}}dy, \,\,\,\,\forall\,\,x\in\mathbb{R}^{N},\,\,\,u\in\mathcal{S},
\end{equation}
where $A(N,s) > 0$ is a suitable normalization constant.
\end{defn}
Also, we describe  the symbol of the fractional Laplacian in the next proposition.
\begin{prop}
    Let $0 < s < 1$ and let $(-\Delta)^{s}:\mathcal{S}\rightarrow L^{2}(\mathbb{R}^{N})$ be the
fractional Laplacian operator defined by \eqref{fraclap}. Then, for any $u\in\mathcal{S}$, we have
\begin{equation*}
     (-\Delta)^{s}u(x)=\mathcal{F}^{-1}\left(|\xi|^{2s}\mathcal{F}\hat{u}\right)(x),\,\,\,\,\xi\in\mathbb{R}^{N}.
\end{equation*}
\end{prop}
\section{Boundedness of the Riesz potential on  global Morrey space}
In this section, for  global Morrey spaces, we formulate  the Adams type Hardy-Littlewood-Sobolev inequality on $\G$, which will be used to obtain the Stein-Weiss-Adams  inequality in Section \ref{SWAinequality}.  Firstly, we show $M^{\lambda}_{p}(\G)$ to $M^{\lambda}_{p}(\G)$ boundedness of the Hardy-Littlewood maximal operator $M_{0}$. In our proofs, many of the tools of the Euclidean setting are generalized to the homogeneous group case. Nevertheless, there are substantive differences and some care must be taken to insure that those proofs still hold.
\begin{thm}[Boundedness of the maximal operator on $M^{\lambda}_{p}(\G)$]\label{MF}
Let $\G$ be a homogeneous group. Assume that $p>1$, $0<\lambda<Q$ and $u\in M_{p}^{\lambda}(\G)$. Then there exists a constant $C$ independent of $u$ such that
\begin{equation}
    \|M_{0}u\|_{M^{\lambda}_{p}(\G)}\leq C\|u\|_{M^{\lambda}_{p}(\G)}.
\end{equation}
\end{thm}
\begin{proof}
Let us set $u(x)=u_{1}(x)+u_{2}(x)$  such that for  $x_{0}\in \G$,
\begin{equation}\label{u1}
 u_{1}(x):=
 \begin{cases}
  u(x),\,\,\,\,\,|x_{0}^{-1}x|\leq 2r,\\
  0,\,\,\,\,\text{otherwise},
 \end{cases}
\end{equation}
and
\begin{equation}\label{u2}
 u_{2}(x):=
 \begin{cases}
  u(x),\,\,\,\,\,|x_{0}^{-1}x|> 2r,\\
  0,\,\,\,\,\text{otherwise}.
 \end{cases}
\end{equation}
Hence, we get
\begin{equation}\label{Jobsh}
\begin{split}
    \int_{B(x_{0},r)}|M_{0}u(x)|^{p}dx&
\leq C\int_{B(x_{0},r)}|M_{0} u_{1}(x)|^{p}dx+C\int_{B(x_{0},r)}|M_{0} u_{2}(x)|^{p}dx\\&=C(J_{1}+J_{2}),
    \end{split}
\end{equation}
where $$M_{0}u(x)=\sup_{\rho>0}\frac{1}{|B(x,\rho)|}\int_{B(x,\rho)}|u(y)|dy.$$
On the one hand, by using the boundedness of $M_{0}:L^{p}(\G)\rightarrow L^{p}(\G)$ with $p>1$ (see \cite[Theorem 3.1]{RSYm}), we have
\begin{equation*}
    J_{1}=\int_{B(x_{0},r)}|M_{0}u_{1}(x)|^{p}dx\leq \int_{\G}|M_{0}u_{1}(x)|^{p}dx\leq C\int_{B(x_{0},2r)}|u(y)|^{p}dy\leq Cr^{\lambda}\|u\|^{p}_{M^{\lambda}_{p}(\G)}.
\end{equation*}
On the other hand, by using H\"{o}lder's inequality and \eqref{ballmeasure}, we compute
\begin{equation*}
    \begin{split}
        \frac{1}{|B(x,\rho)|}\int_{B(x,\rho)}|u_{2}(y)|dy&\leq C\rho^{-Q}\int_{B(x,\rho)\cap \left(\G\setminus B(x_{0},2r)\right)}|u(y)|dy\\&
        \leq C\rho^{-Q}\int_{B(x,\rho)}|u(y)|dy\\&
        \leq C \rho^{-Q+\frac{Q}{p'}}\left(\int_{B(x,\rho)}|u(y)|^{p}dy\right)^{\frac{1}{p}}\\&
        = C \rho^{-\frac{Q-\lambda}{p}}\left(\frac{1}{\rho^{\lambda}}\int_{B(x,\rho)}|u(y)|^{p}dy\right)^{\frac{1}{p}}.
    \end{split}
\end{equation*}
By using \eqref{u2}, $|y^{-1}x|\leq \rho$ and $|x^{-1}_{0}x|\leq r$ (in \eqref{Jobsh}), we have
\begin{equation*}
    2r<|y^{-1}x_{0}|=|y^{-1}xx^{-1}x_{0}|\leq |y^{-1}x|+|x^{-1}_{0}x|\leq \rho+r,
\end{equation*}
that is, $r\leq \rho$, we arrive at
\begin{equation*}
\begin{split}
     \frac{1}{|B(x,\rho)|}\int_{B(x,\rho)}|u_{2}(x)|dx&\leq C \rho^{-\frac{Q-\lambda}{p}}\left(\frac{1}{\rho^{\lambda}}\int_{B(x,\rho)}|u(y)|^{p}dy\right)^{\frac{1}{p}}\\&
  \leq C r^{-\frac{Q-\lambda}{p}}\|u\|_{M^{\lambda}_{p}(\G)}.
\end{split}
\end{equation*}
Hence, we get
\begin{equation*}
    M_{0}u_{2}(x)=\sup_{\rho>0}\frac{1}{|B(x,\rho)|}\int_{B(x,\rho)}|u_{2}(y)|dy\leq C r^{-\frac{Q-\lambda}{p}}\|u\|_{M^{\lambda}_{p}(\G)}.
\end{equation*}
It yields
\begin{equation*}
\begin{split}
    J_{2}=\int_{|x_{0}^{-1}x|<r}|M_{0} u_{2}(x)|^{p}dx\leq C\rho^{-Q+\lambda}\|u\|^{p}_{M^{\lambda}_{p}(\G)}\int_{|x_{0}^{-1}x|<r}dx\leq Cr^{\lambda}\|u\|^{p}_{M^{\lambda}_{p}(\G)}.
\end{split}
\end{equation*}
Finally,  we obtain
\begin{equation}\label{bounhl}
\int_{B(x_{0},r)}|M_{0}u(x)|^{p}dx\leq C(J_{1}+J_{2})\leq Cr^{\lambda}\|u\|^{p}_{M^{\lambda}_{p}(\G)},
\end{equation}
completing the proof.
\end{proof}
\begin{thm}[Adams type Hardy-Littlewood-Sobolev inequality]\label{stein-weiss3}
Let $\mathbb{G}$ be a homogeneous group of homogeneous dimension $Q$ and let $|\cdot|$ be a quasi-norm on $\mathbb{G}$.
Let  $0<\gamma<Q$, $1<p<\frac{Q}{\gamma}$, $1<p<q<\infty$, $0<\lambda<Q-\gamma p$ and $\frac{1}{q}=\frac{1}{p}-\frac{\gamma}{Q-\lambda}$, where $\frac{1}{p}+\frac{1}{p'}=1$ and $\frac{1}{q}+\frac{1}{q'}=1$. Then, we have
\begin{equation}\label{HLS}
\|I_{\gamma}u\|_{M^{\lambda}_{q}(\G)}\leq C\|u\|_{M^{\lambda}_{p}(\G)},
\end{equation}
where $C$ is a positive constant independent of $u$.
\end{thm}
\begin{rem}
In the case $\lambda=0$, Theorem \ref{stein-weiss3} implies the Hardy-Littlewood-Sobolev inequality on homogeneous Lie groups for Lebesgue spaces, which was investigated in \cite{KRS}.
\end{rem}
\begin{rem}
In the Abelian (Euclidean) case ${\mathbb G}=(\mathbb R^{N},+)$, that is, for $Q=N$ and $|\cdot|=|\cdot|_{E}$ ($|\cdot|_{E}$ is the Euclidean distance),  Theorem \ref{stein-weiss3} covers the classical result from \cite{A}.

\end{rem}

\begin{proof}
The proof of this theorem is based  on Hedberg's trick. Let us decompose the Riesz potential operator in the following form:
\begin{equation*}
\begin{split}
    |I_{\gamma}u(x)|&\leq \int_{\G}|y^{-1}x|^{\gamma-Q}|u(y)|dy\\&
    =\int_{|y^{-1}x|\leq \rho}|y^{-1}x|^{\gamma-Q}|u(y)|dy+\int_{|y^{-1}x|> \rho}|y^{-1}x|^{\gamma-Q}|u(y)|dy\\&
    =J_{1}(x)+J_{2}(x).
\end{split}
\end{equation*}
Firstly, let us consider $J_{1}(x)$. A straightforward computation gives
\begin{equation*}
    \begin{split}
     J_{1}(x)&=\int_{|y^{-1}x|\leq \rho}|y^{-1}x|^{\gamma-Q}|u(y)|dy\\&
     =\sum_{k=-\infty}^{0}\int_{2^{k-1}\rho\leq|y^{-1}x|\leq 2^{k}\rho}|y^{-1}x|^{\gamma-Q}|u(y)|dy\\&
     \stackrel{\gamma< Q}\leq \sum_{k=-\infty}^{0}(2^{k-1}\rho)^{\gamma-Q}\int_{2^{k-1}\rho\leq|y^{-1}x|\leq 2^{k}\rho}|u(y)|dy\\&
     \leq \sum_{k=-\infty}^{0}(2^{k-1}\rho)^{\gamma-Q}\int_{|y^{-1}x|\leq 2^{k}\rho}|u(y)|dy\\&
     =\sum_{k=-\infty}^{0}(2^{k-1}\rho)^{\gamma-Q}\frac{|B(x,2^{k}\rho)|}{|B(x,2^{k}\rho)|}\int_{|y^{-1}x|\leq 2^{k}\rho}|u(y)|dy\\&
     \leq C\rho^{\gamma-Q}\left(\sum_{k=-\infty}^{0}(2^{k})^{\gamma-Q}|B(x,2^{k}\rho)|\right)\left(M_{0}u\right)(x)
      \end{split}
\end{equation*}
\begin{equation*}
    \begin{split}
     &\stackrel{\eqref{ballmeasure}}\leq C\rho^{\gamma-Q}\left(\sum_{k=-\infty}^{0}(2^{k})^{\gamma-Q}(2^{k}\rho)^{Q}\right)\left(M_{0}u\right)(x)\\&
     =C\rho^{\gamma}\left(\sum_{k=-\infty}^{0}(2^{k})^{\gamma}\right)\left(M_{0}u\right)(x)\\&
     \stackrel{\gamma>0}\leq C \rho^{\gamma}\left(M_{0}u\right)(x),
    \end{split}
\end{equation*}
and similarly, for $J_{2}(x)$, one has
\begin{equation*}
    \begin{split}
     J_{2}(x)&=\int_{|y^{-1}x|> \rho}|y^{-1}x|^{\gamma-Q}|u(y)|dy\\&
     =\sum^{+\infty}_{k=1}\int_{2^{k-1}\rho\leq|y^{-1}x|\leq 2^{k}\rho}|y^{-1}x|^{\gamma-Q}|u(y)|dy\\&
   \stackrel{\gamma< Q}\leq C\sum^{+\infty}_{k=1}(2^{k}\rho)^{\gamma-Q}\int_{2^{k-1}\rho\leq|y^{-1}x|\leq 2^{k}\rho}|u(y)|dy\\&
\leq C\sum^{+\infty}_{k=1}(2^{k}\rho)^{\gamma-Q}\frac{|B(x,2^{k}\rho)|^{1-\frac{Q-\lambda}{Qp}}}{|B(x,2^{k}\rho)|^{1-\frac{Q-\lambda}{Qp}}}\int_{|y^{-1}x|\leq 2^{k}\rho}|u(y)|dy\\&
   \leq C\left(\sum^{+\infty}_{k=1}(2^{k}\rho)^{\gamma-Q}(2^{k}\rho)^{Q-\frac{Q-\lambda}{p}}\right)\left(M_{\frac{Q-\lambda}{Qp}}u \right)(x)\\&
   =C\rho^{\gamma-\frac{Q-\lambda}{p}}\left(\sum^{+\infty}_{k=1}(2^{k})^{\gamma-\frac{Q-\lambda}{p}}\right)\left(M_{\frac{Q-\lambda}{Qp}}u \right)(x)\\&
   \stackrel{\lambda<Q-\gamma p}\leq
   C\rho^{\gamma-\frac{Q-\lambda}{p}}\left(M_{\frac{Q-\lambda}{Qp}}u \right)(x).
    \end{split}
\end{equation*}
By choosing $\rho=\left(\frac{M_{\frac{Q-\lambda}{Qp}}u(x)}{M_{0}u(x)}\right)^{\frac{p}{Q-\lambda}}$, we get
\begin{equation*}
    |I_{\gamma}u(x)|\leq J_{1}+J_{2}=C \left(M_{\frac{Q-\lambda}{Qp}}u(x) \right)^{\frac{p \gamma}{Q-\lambda}}\left(M_{0}u(x) \right)^{1-\frac{p \gamma}{Q-\lambda}}.
\end{equation*}
Let us estimate the following term:
\begin{equation*}
    \begin{split}
     \frac{1}{|B(x,r)|^{1-\frac{Q-\lambda}{Qp}}}\int_{B(x,r)}|u(y)|dy\leq Cr^{-Q+\frac{Q-\lambda}{p}+\frac{Q}{p'}}\left(\int_{B(x,r)}|u(y)|^{p}dy\right)^{\frac{1}{p}}\leq C\|u\|_{M^{\lambda}_{p}(\G)},
    \end{split}
\end{equation*}
that is,
\begin{equation*}
     M_{\frac{Q-\lambda}{Qp}}u(x) \leq C\|u\|_{M^{\lambda}_{p}(\G)}.
\end{equation*}
Thus, we obtain the pointwise estimate:
\begin{equation*}
    |I_{\gamma}u(x)|\leq \left(M_{0}u(x) \right)^{1-\frac{p \gamma}{Q-\lambda}}\|u\|^{\frac{p\gamma}{Q-\lambda}}_{M^{\lambda}_{p}(\G)}.
\end{equation*}
By using the boundedness of $M_{0}:M_{p}^{\lambda}(\G)\rightarrow M_{p}^{\lambda}(\G)$ from Theorem \ref{MF}, we have
\begin{equation*}
\begin{split}
      \int_{B(x,r)}|I_{\gamma}u(x)|^{q}dx&\leq C\|u\|^{\frac{qp\gamma}{Q-\lambda}}_{M^{\lambda}_{p}(\G)}\int_{B(x,r)}\left(M_{0}u(x) \right)^{q-\frac{qp \gamma}{Q-\lambda}}dx\\&
      =C\|u\|^{\frac{qp\gamma}{Q-\lambda}}_{M^{\lambda}_{p}(\G)}\int_{B(x,r)}\left(M_{0}u(x) \right)^{p}dx\\&
      \stackrel{\lambda<Q-\gamma p}\leq Cr^{\lambda}\|u\|^{\frac{qp\gamma}{Q-\lambda}+p}_{M^{\lambda}_{p}(\G)}\\&
      \stackrel{\frac{1}{q}=\frac{1}{p}-\frac{\gamma}{Q-\lambda}}=Cr^{\lambda}\|u\|^{q}_{M^{\lambda}_{p}(\G)},
\end{split}
\end{equation*}
completing the proof.
\end{proof}

\section{Adams type Stein-Weiss inequality} \label{SWAinequality}

Here we state the Stein-Weiss-Adams inequality on the global  Morrey space.

\begin{thm}[Adams type Stein-Weiss inequality]\label{stein-weissthm}
Let $\mathbb{G}$ be a homogeneous group of homogeneous dimension $Q$ such that $0<\gamma <Q$ and let $|\cdot|$ be a quasi-norm on $\mathbb{G}$.
Let  $\alpha,\beta\in \mathbb{R}$, $0\leq \alpha+\beta\leq\gamma<Q$, $1<p<\frac{Q}{\gamma-\alpha-\beta}$, $\frac{1}{q}=\frac{1}{p}-\frac{\gamma-\alpha-\beta}{Q-\lambda}$,  $\alpha<\frac{Q}{p'}$, $\beta<\frac{Q-\lambda}{q}$ and $0<\lambda<Q-(\gamma-\alpha-\beta)p$. Then, for all  $|\cdot|^{\alpha}u\in M^{\lambda}_{p}(\G)$, we have
\begin{equation}\label{stein-weiss}
\||\cdot|^{-\beta}I_{\gamma}u\|_{M^{\lambda}_{q}(\G)}\leq C\||\cdot|^{\alpha}u\|_{M^{\lambda}_{p}(\G)},
\end{equation}
where $C$ is a positive constant independent of $u$.
\end{thm}

\begin{rem}
 Inequality \eqref{stein-weiss} with $\alpha=\beta=0$ gives the Adams type Hardy-Littlewood-Sobolev inequality \eqref{HLS}.
\end{rem}
\begin{rem}\label{remeuc}
To the best of our knowledge, in the Abelian (Euclidean) case ${\mathbb G}=(\mathbb R^{N},+)$, that is, with $Q=N$ and $|\cdot|=|\cdot|_{E}$ ($|\cdot|_{E}$ is the Euclidean distance), the inequality \eqref{stein-weiss} is already new.
\end{rem}
\begin{proof}[Proof of Theorem \ref{stein-weissthm}]
In view of Proposition \ref{prop_quasi_norm}, without loss of generality, one can assume that the quasi-norm $|\cdot|$ is a norm. Let us decompose the Riesz potential in the following way:
\begin{equation*}
\begin{split}
    I_{\gamma}u(x)&= \int_{\G}|y^{-1}x|^{\gamma-Q}u(y)dy\\&
    =\int_{|y|<\frac{|x|}{2}}|y^{-1}x|^{\gamma-Q}u(y)dy+\int_{\frac{|x|}{2}\leq|y|\leq 2|x|}|y^{-1}x|^{\gamma-Q}u(y)dy+\int_{2|x|<|y|}|y^{-1}x|^{\gamma-Q}u(y)dy,
\end{split}
\end{equation*}
and hence, we get
\begin{equation*}
\begin{split}
    |I_{\gamma}u(x)|&\leq  \int_{|y|<\frac{|x|}{2}}|y^{-1}x|^{\gamma-Q}|u(y)|dy+\int_{\frac{|x|}{2}\leq|y|\leq 2|x|}|y^{-1}x|^{\gamma-Q}|u(y)|dy\\&
    +\int_{2|x|<|y|}|y^{-1}x|^{\gamma-Q}|u(y)|dy.
\end{split}
\end{equation*}
Now we divide the proof into 3 steps.

\textbf{Step 1.} Let us consider the first term $J_{1}(x):=\int_{|y|<\frac{|x|}{2}}|y^{-1}x|^{\gamma-Q}|u(y)|dy$.
Firstly, we calculate:
\begin{equation}\label{ocen1}
    \begin{split}
        \int_{|y|<|x|}|u(y)|dy&=\sum_{k=0}^{\infty}\int_{2^{-k-1}|x|\leq |y|\leq 2^{-k}|x|}|u(y)|dy\\&
        =\sum_{k=0}^{\infty}\int_{2^{-k-1}|x|\leq |y|\leq 2^{-k}|x|}|u(y)||y|^{\alpha}|y|^{-\alpha}dy\\&
        \leq C \sum_{k=0}^{\infty}(2^{-k}|x|)^{-\alpha}\int_{2^{-k-1}|x|\leq |y|\leq 2^{-k}|x|}|u(y)||y|^{\alpha}dy\\&
        \leq C \sum_{k=0}^{\infty}(2^{-k}|x|)^{-\alpha}\int_{|y|\leq 2^{-k}|x|}|u(y)||y|^{\alpha}dy\\&
        \leq C\sum_{k=0}^{\infty}(2^{-k}|x|)^{-\alpha+\frac{Q}{p'}}\left(\int_{|y|\leq 2^{-k}|x|}|u(y)|^{p}|y|^{\alpha p}dy\right)^{\frac{1}{p}}\\&
        \leq C\sum_{k=0}^{\infty}(2^{-k}|x|)^{-\alpha+\frac{Q}{p'}+\frac{\lambda}{p}}\||\cdot|^{\alpha}u\|_{LM^{\lambda}_{p}(\G)}\\&
        \leq C|x|^{\frac{\lambda}{p}}\sum_{k=0}^{\infty}(2^{-k}|x|)^{-\alpha+\frac{Q}{p'}}\||\cdot|^{\alpha}u\|_{LM^{\lambda}_{p}(\G)}\\&
        \stackrel{\alpha<\frac{Q}{p'}}\leq C|x|^{-\alpha+Q-\frac{Q-\lambda}{p}}\||\cdot|^{\alpha }u\|_{LM^{\lambda}_{p}(\G)}.
    \end{split}
\end{equation}
Consider the following integral
\begin{equation*}
    \sup_{r>0}\frac{1}{r^{\lambda}}\int_{|y^{-1}x|<r}|u(y)|^{p}dy\leq \sup_{x\in \G}\sup_{r>0}\frac{1}{r^{\lambda}}\int_{|y^{-1}x|<r}|u(y)|^{p}dy=\|u\|^{p}_{M^{\lambda}_{p}(\G)},
\end{equation*}
and by taking $x=e$ where $e$ is an identity element of $\G$, we get
\begin{equation}\label{emd}
    \|u\|^{p}_{LM^{\lambda}_{p}(\G)}=\sup_{r>0}\frac{1}{r^{\lambda}}\int_{|y|<r}|u(y)|^{p}dy\leq \sup_{x\in \G}\sup_{r>0}\frac{1}{r^{\lambda}}\int_{|y^{-1}x|<r}|u(y)|^{p}dy=\|u\|^{p}_{M^{\lambda}_{p}(\G)}.
\end{equation}
That is, we have the embedding $M^{\lambda}_{p}(\G)\subset LM^{\lambda}_{p}(\G)$.
By combining \eqref{emd} with \eqref{ocen1}, we obtain
\begin{equation}\label{ocenwit}
\begin{split}
    \int_{|y|<|x|}|u(y)|dy&=\sum_{k=0}^{\infty}\int_{2^{-k-1}|x|\leq |y|\leq 2^{-k}|x|}|u(y)|dy\\&
    \leq C|x|^{-\alpha+Q-\frac{Q-\lambda}{p}}\||\cdot|^{\alpha }u\|_{LM^{\lambda}_{p}(\G)}\\&
    \leq C |x|^{-\alpha+Q-\frac{Q-\lambda}{p}}\||\cdot|^{\alpha }u\|_{M^{\lambda}_{p}(\G)}.
\end{split}
\end{equation}
By using Proposition \ref{prop_quasi_norm} and the properties of the (quasi-)norm with  $|y|<\frac{|x|}{2}$, we have
$$|x|\leq |y^{-1} x|+|y| \leq |y^{-1} x|+\frac{|x|}{2}.$$
Then for any $Q>\gamma>0$, it is clear that
$$2^{\gamma-Q}|x|^{\gamma-Q}\geq |y^{-1} x|^{\gamma-Q}.$$
By combining this with \eqref{ocenwit}, we have
\begin{equation*}
    \begin{split}
        J_{1}(x)&=\int_{|y|<\frac{|x|}{2}}|y^{-1}x|^{\gamma-Q}|u(y)|dy\\&
        \leq C |x|^{\gamma-Q}\int_{|y|<\frac{|x|}{2}}|u(y)|dy\\&
        \leq C |x|^{\gamma-Q}\int_{|y|<|x|}|u(y)|dy\\&
        \stackrel{\alpha<\frac{Q}{p'}}\leq C|x|^{\gamma-\alpha-\frac{Q-\lambda}{p}}\||\cdot|^{\alpha}u\|_{M^{\lambda}_{p}(\G)}.
    \end{split}
\end{equation*}
By using $\frac{1}{q}=\frac{1}{p}-\frac{\gamma-\alpha-\beta}{Q-\lambda}$ and \eqref{EQ:polar}, we obtain
\begin{equation}\label{okonocen}
    \begin{split}
      &\int_{|z^{-1}x|\leq r}|x|^{-\beta q}\left(\int_{|y|<\frac{|x|}{2}}|y^{-1}x|^{\gamma-Q}|u(y)|dy\right)^{q}dx=   \int_{|z^{-1}x|\leq r}|x|^{-\beta q}J^{q}_{1}(x)dx\\&
      \leq C\||\cdot|^{\alpha}u\|^{q}_{M^{\lambda}_{p}(\G)}\int_{|z^{-1}x|\leq r}|x|^{\left(-\beta+\gamma-\alpha-\frac{Q-\lambda}{p}\right) q}dx\\&
      =C\||\cdot|^{\alpha}u\|^{q}_{M^{\lambda}_{p}(\G)}\int_{|z^{-1}x|\leq r}|x|^{\lambda-Q}dx.
    \end{split}
\end{equation}
To estimate the above integral we consider two cases $|z|\leq \theta r$ and $|z|\geq \theta r,$ where $\theta>1$. Firstly, let us calculate the case $|z|\leq  \theta r$. With $|z^{-1}x|\leq r$  we have $|x|\leq |z^{-1}x|+|z|\leq (\theta +1)r$, that is, $B(z,r)\subset B(e,(\theta +1) r)$. By using this fact with polar decomposition \eqref{EQ:polar}, we get
\begin{equation*}
  \int_{|z^{-1}x|\leq r}|x|^{\lambda-Q}dx \leq  \int_{|x|\leq (\theta +1)r}|x|^{\lambda-Q}dx=C\int_{0}^{(\theta +1)r}t^{\lambda-Q}t^{Q-1}dt\leq Cr^{\lambda}.
\end{equation*}
Now we compute the case $|z|\geq \theta r$. Hence, we obtain $\theta r \leq |z|\leq  |z^{-1}x|+|x|\leq r+|x|$, that is, $|x|\geq (\theta-1)r$, where $\theta>1$. It implies
\begin{equation*}
    \int_{|z^{-1}x|\leq r}|x|^{\lambda-Q}dx\stackrel{\lambda<Q}\leq (\theta-1)^{\lambda-Q}r^{\lambda-Q}\int_{|z^{-1}x|\leq r}dx\leq Cr^{\lambda-Q}r^{Q}\leq Cr^{\lambda}.
\end{equation*}
By combining these facts, we have
\begin{equation}\label{ocennashar}
     \int_{|z^{-1}x|\leq r}|x|^{\lambda-Q}dx\leq Cr^{\lambda}.
\end{equation}
 Thus,  the estimate \eqref{okonocen} yields
\begin{equation*}
    \begin{split}
      \int_{|z^{-1}x|\leq r}|x|^{-\beta q}\left(\int_{|y|<\frac{|x|}{2}}|y^{-1}x|^{\gamma-Q}|u(y)|dy\right)^{q}dx&\leq C\||\cdot|^{\alpha}u\|^{q}_{M^{\lambda}_{p}(\G)}\int_{|z^{-1}x|\leq r}|x|^{\lambda-Q}dx\\&
      \leq Cr^{\lambda} \||\cdot|^{\alpha}u\|^{q}_{M^{\lambda}_{p}(\G)}.
    \end{split}
\end{equation*}
\textbf{Step 2.} Let us consider the last term $J_{3}(x):=\int_{2|x|<|y|}|y^{-1}x|^{\gamma-Q}|u(y)|dy$. Since $\beta<\frac{Q-\lambda}{q},$ we have
$$\frac{1}{q}=\frac{1}{p}-\frac{\gamma-\alpha-\beta}{Q-\lambda}<\frac{1}{p}+\frac{1}{q}-\frac{\gamma-\alpha}{Q-\lambda},$$
and hence
\begin{equation}\label{ocen7}
    \frac{Q-\lambda}{p}-\gamma+\alpha>0.
\end{equation}
In a similar way as in the previous step, consider the following integral:
\begin{equation}\label{ocen2}
    \begin{split}
        \int_{|y|>2|x|}|y|^{\gamma-Q}|u(y)|dy&\leq\sum_{k=0}^{\infty}\int_{2^{k-1}|x|\leq |y|\leq 2^{k}|x|}|y|^{\gamma-Q-\alpha}|u(y)||y|^{\alpha}dy\\&
        \leq C\sum_{k=0}^{\infty}(2^{k}|x|)^{\gamma-Q-\alpha}\int_{2^{k-1}|x|\leq |y|\leq 2^{k}|x|}|u(y)||y|^{\alpha}dy\\&
        \leq C\sum_{k=0}^{\infty}(2^{k}|x|)^{\gamma-Q-\alpha+\frac{Q}{p'}}\left(\frac{ (2^{k}|x|)^{\lambda}}{(2^{k}|x|)^{\lambda}}\int_{|y|\leq 2^{k}|x|}|u(y)|^{p}|y|^{\alpha p}dy\right)^{\frac{1}{p}}\\&
        \leq C\sum_{k=0}^{\infty}(2^{k}|x|)^{\gamma-\alpha-\frac{Q-\lambda}{p}}\||\cdot|^{\alpha}u\|_{LM^{\lambda}_{p}(\G)}\\&
        \stackrel{\eqref{ocen7}}\leq C|x|^{\gamma-\alpha-\frac{Q-\lambda}{p}}\||\cdot|^{\alpha}u\|_{LM^{\lambda}_{p}(\G)}\\&
        \stackrel{\eqref{emd}}\leq C|x|^{\gamma-\alpha-\frac{Q-\lambda}{p}}\||\cdot|^{\alpha}u\|_{M^{\lambda}_{p}(\G)}.
    \end{split}
\end{equation}
From $2|x|\leq |y|$,  we calculate
$$|y|=|y^{-1}|=|y^{-1} x x^{-1}|\leq |y^{-1} x|+|x|\leq |y^{-1} x|+\frac{|y|}{2},$$
that is,
$$\frac{|y|}{2}\leq |y^{-1} x|.$$
By combining it with \eqref{ocen2}, we obtain
\begin{equation*}
    \begin{split}
        J_{3}(x)=\int_{2|x|<|y|}|y^{-1}x|^{\gamma-Q}|u(y)|dy&\leq C \int_{|y|>2|x|}|y|^{\gamma-Q}|u(y)|dy\\&
        \leq C|x|^{\gamma-\alpha-\frac{Q-\lambda}{p}}\||\cdot|^{\alpha}u\|_{M^{\lambda}_{p}(\G)}.
    \end{split}
\end{equation*}
Hence, from $\frac{1}{q}=\frac{1}{p}-\frac{\gamma-\alpha-\beta}{Q-\lambda}$ and \eqref{ocennashar}, we establish
\begin{equation*}
    \begin{split}
        &\int_{|z^{-1}x|\leq r}|x|^{-\beta q}\left(\int_{2|x|<|y|}|y^{-1}x|^{\gamma-Q}|u(y)|dy\right)^{q}dx=   \int_{|z^{-1}x|\leq r}|x|^{-\beta q}J^{q}_{3}(x)dx\\&
        \leq C\||\cdot|^{\alpha}u\|^{q}_{M^{\lambda}_{p}(\G)}\int_{|z^{-1}x|\leq r}|x|^{\left(-\beta+\gamma-\alpha-\frac{Q-\lambda}{p}\right) q}dx\\&
        =C\||\cdot|^{\alpha}u\|^{q}_{M^{\lambda}_{p}(\G)}\int_{|z^{-1}x|\leq r}|x|^{\lambda-Q}dx\\&
        \stackrel{\eqref{ocennashar}}\leq Cr^{\lambda}\||\cdot|^{\alpha}u\|^{q}_{M^{\lambda}_{p}(\G)}.
    \end{split}
\end{equation*}

\textbf{Step 3.} In this step we consider the function $J_{2}(x):=\int_{\frac{|x|}{2}\leq|y|\leq 2|x|}|y^{-1}x|^{\gamma-Q}|u(y)|dy$. Let us first consider the case $\gamma>\alpha+\beta\geq 0.$ Then, by using $\frac{|x|}{2}\leq|y|\leq 2|x|$, we get
\begin{equation*}
    |y^{-1}x|^{\alpha+\beta}\leq C (|x|^{\alpha+\beta}+|y|^{\alpha+\beta})\leq C|y|^{\alpha+\beta}\leq C|x|^{\beta}|y|^{\alpha}.
\end{equation*}
Therefore, we have
\begin{equation*}
\begin{split}
    |x|^{-\beta}J_{2}(x)&=|x|^{-\beta}\int_{\frac{|x|}{2}\leq|y|\leq 2|x|}|y^{-1}x|^{\gamma-Q}|u(y)|dy\\&
    = \int_{\frac{|x|}{2}\leq|y|\leq 2|x|}|x|^{-\beta}|y^{-1}x|^{\gamma-Q}|y^{-1}x|^{-\alpha-\beta}|y^{-1}x|^{\alpha+\beta}|u(y)|dy \\&
    \leq C\int_{\frac{|x|}{2}\leq|y|\leq 2|x|}|x|^{-\beta}|y^{-1}x|^{\gamma-\alpha-\beta-Q}|x|^{\beta}|y|^{\alpha}|u(y)|dy\\&
    =C\int_{\frac{|x|}{2}\leq|y|\leq 2|x|}|y^{-1}x|^{\gamma-\alpha-\beta-Q}|\tilde{u}(y)|dy\\&
    \leq C\int_{\G}|y^{-1}x|^{\gamma_{1}-Q}|\tilde{u}(y)|dy,
\end{split}
\end{equation*}
where $\tilde{u}(x)=|x|^{\alpha}u(x)$ and $\gamma_{1}=\gamma-\alpha-\beta$. By assumption, we have $Q> \gamma \stackrel{\alpha+\beta\geq 0}\geq\gamma_{1}=\gamma-\alpha-\beta>0$ (that is $0<\gamma_{1}<Q$),  $0<\lambda<Q-\gamma_{1}p$ and $\frac{1}{q}=\frac{1}{p}-\frac{\gamma_{1}}{Q-\lambda}$, so one can apply Theorem \ref{stein-weiss3} to get
\begin{equation*}
\begin{split}
     \int_{|z^{-1}x|\leq r}|x|^{-\beta q}J^{q}_{2}(x)dx&\leq C\int_{|z^{-1}x|\leq r}\left(\int_{\G}|y^{-1}x|^{\gamma_{1}-Q}|\tilde{u}(y)|dy\right)^{q}dx\\&
    \leq Cr^{\lambda}\|\tilde{u}\|^{q}_{M^{\lambda}_{p}(\G)}\\&
    =Cr^{\lambda}\||\cdot|^{\alpha}u\|^{q}_{M^{\lambda}_{p}(\G)}.
\end{split}
\end{equation*}
Let us now focus on the case $\gamma=\alpha+\beta>0$. So, we have $p=q>1$. Firstly, consider two cases $|z|\leq \theta r$ and $|z|\geq \theta r$, where $\theta>1$.
Assume that $|z|\leq \theta r$ and $|z^{-1}x|\leq r$, then $|x|\leq |z^{-1}x|+|z|\leq(\theta+1)r$, that is, $B(z,r)\subset B(e,(\theta+1)r)$, where $\theta>1$. We compute
\begin{equation*}
    \begin{split}
        \int_{|z^{-1}x|\leq r}&|x|^{-\beta p}J^{p}_{2}(x)dx\leq \int_{|x|\leq (\theta+1)r}|x|^{-\beta p}J^{p}_{2}(x)dx\\&
        =\int_{|x|\leq (\theta+1)r}|x|^{-\beta p}\left(\int_{\frac{|x|}{2}\leq|y|\leq 2|x|}|y^{-1}x|^{\gamma-Q}|u(y)|dy\right)^{p}dx\\&
        \leq C\int_{|x|\leq (\theta+1)r}|x|^{-(\alpha+\beta) p}\left(\int_{\frac{|x|}{2}\leq|y|\leq 2|x|}|y^{-1}x|^{\gamma-Q}|\tilde{u}(y)|dy\right)^{p}dx\\&
        =C\int_{|x|\leq (\theta+1)r}|x|^{-\gamma p}\left(\int_{\frac{|x|}{2}\leq|y|\leq 2|x|}|y^{-1}x|^{\gamma-Q}|\tilde{u}(y)|dy\right)^{p}dx\\&
        =C\sum_{k=0}^{\infty}\int_{2^{-k-1}(\theta+1)r\leq |x|\leq 2^{-k}(\theta+1)r}|x|^{-\gamma p}\left(\int_{\frac{|x|}{2}\leq|y|\leq 2|x|}|y^{-1}x|^{\gamma-Q}|\tilde{u}(y)|dy\right)^{p}dx,
    \end{split}
\end{equation*}
where $\tilde{u}(x)=|x|^{\alpha}u(x)$.
From $\frac{|x|}{2}\leq|y|\leq 2|x|$ and $2^{-k-1}(\theta+1)r\leq |x|\leq 2^{-k}(\theta+1)r$, we have
\begin{equation}
   2^{-k-2}(\theta+1)r\leq |y|\leq 2^{-k+1}(\theta+1)r,
\end{equation}
and
\begin{equation*}
    |y^{-1}x|\leq |x|+|y|\leq 3|x|\leq 3\cdot2^{-k}(\theta+1)r.
\end{equation*}
From $1<p\leq q<\infty$ for $1+\frac{1}{q}=\frac{1}{p}+\frac{1}{\zeta}$ it follows that $\zeta\in[1,+\infty]$ and $\zeta=1$ by taking $p=q$. Combining these with the Young inequality for the Lebesgue space we get
\begin{equation*}
    \begin{split}
        &\int_{|z^{-1}x|\leq r}|x|^{-\beta p}J^{p}_{2}(x)dx
       \\& \leq C\sum_{k=0}^{\infty}\int_{2^{-k-1}(\theta+1)r\leq |x|\leq 2^{-k}(\theta+1)r}|x|^{-\gamma p}\left(\int_{\frac{|x|}{2}\leq|y|\leq 2|x|}|y^{-1}x|^{\gamma-Q}|\tilde{u}(y)|dy\right)^{p}dx \\&
        \leq C\sum_{k=0}^{\infty}(2^{-k}(\theta+1)r)^{-\gamma p}
        \\& \times \int_{2^{-k-1}(\theta+1)r\leq |x|\leq 2^{-k}(\theta+1)r}\left(\int_{2^{-k-2}(\theta+1)r\leq|y|\leq 3\cdot2^{-k}(\theta+1)r}|y^{-1}x|^{\gamma-Q}|\tilde{u}(y)|dy\right)^{p}dx\\&
        \leq C\sum_{k=0}^{\infty}(2^{-k}(\theta+1)r)^{-\gamma p}\int_{\G}\left(\int_{\G}|y^{-1}x|^{\gamma-Q}|\tilde{u}(y)|\chi_{\{2^{-k-2}(\theta+1)r\leq|y|\leq 3\cdot2^{-k}(\theta+1)r\}}dy\right)^{p}dx\\&
        =C\sum_{k=0}^{\infty}(2^{-k}r)^{-\gamma p}\||\cdot|^{\gamma-Q}\ast[\tilde{u}\chi_{\{2^{-k-2}(\theta+1)r\leq|\cdot|\leq 3\cdot2^{-k}(\theta+1)r\}}]\|^{p}_{L^{p}(\G)}\\&
         \leq C\sum_{k=0}^{\infty}(2^{-k}(\theta+1)r)^{-\gamma p}\||\cdot|^{\gamma-Q}\chi_{\{|\cdot|\leq3\cdot2^{-k}(\theta+1)r\}}\|^{p}_{L^{1}(\G)}\|\tilde{u}\chi_{\{2^{-k-2}(\theta+1)r\leq|\cdot|\leq 3\cdot2^{-k}(\theta+1)r\}}\|^{p}_{L^{p}(\G)}
          \end{split}
\end{equation*}
        \begin{equation*}
    \begin{split}
        &
        \stackrel{\eqref{ocennashar}}\leq C\sum_{k=0}^{\infty}(2^{-k}r)^{-\gamma p}(2^{-k}r)^{\gamma p}\|\tilde{u}\chi_{\{2^{-k-2}(\theta+1)r\leq|\cdot|\leq 3\cdot2^{-k}(\theta+1)r\}}\|^{p}_{L^{p}(\G)}\\&
        \leq C\sum_{k=0}^{\infty}\int_{|y|\leq 3\cdot2^{-k}(\theta+1)r}|\tilde{u}(y)|^{p}dy\\&
        =C\sum_{k=0}^{\infty}\frac{(3\cdot2^{-k}(\theta+1)r)^{\lambda}}{(3\cdot2^{-k}(\theta+1)r)^{\lambda}}\int_{|y|\leq 3\cdot2^{-k}(\theta+1)r}|\tilde{u}(y)|^{p}dy\\&
        \leq C\|\tilde{u}\|^{p}_{LM^{\lambda}_{p}(\G)}\sum_{k=0}^{\infty}(2^{-k}r)^{\lambda}\\&
        \stackrel{\eqref{ocenwit}}\leq Cr^{\lambda}\|\tilde{u}\|^{p}_{M^{\lambda}_{p}(\G)}\\&
        =Cr^{\lambda}\||\cdot|^{\alpha}u\|^{p}_{M^{\lambda}_{p}(\G)}.
    \end{split}
\end{equation*}
Consider the case $|z|>\theta r$, where $\theta>1$. By using $|z^{-1}x|\leq r$, we have
\begin{equation}\label{ocenemb3}
    \theta r<|z|\leq |z^{-1}x|+|x|\leq r+|x|\,\,\,\,\Rightarrow \,\,\,|x|\geq (\theta-1)r,\,\,\text{where}\,\,\,\theta>1,
\end{equation}
and also by $|y|\leq 2|x|$, we get
\begin{equation}\label{ocenemb1}
    |y^{-1}x|\leq |x|+|y|\leq 3|x|\leq 3(|z^{-1}x|+|z|)\leq 3(r+|z|)\,\,\,\Rightarrow\,\,\,B(e,2|x|)\subset B(x,3(r+|z|)),
\end{equation}
where $e$ is an identity element of $\G$. In addition, we have
\begin{equation}\label{ocenemb2}
    |z|\leq |z^{-1}x|+|x|\leq r+|x|\,\,\,\,\Rightarrow\,\,\,\,|z|\leq r+|x|.
\end{equation}
By using these facts, for $\tilde{u}(x)=|x|^{\alpha}u(x)$, we calculate
\begin{equation*}
    \begin{split}
        J_{2}(x)&=\int_{\frac{|x|}{2}\leq |y|\leq 2|x|}|y^{-1}x|^{\gamma-Q}u(y)dy\\&
        \leq C|x|^{-\alpha}\int_{\frac{|x|}{2}\leq |y|\leq 2|x|}|y^{-1}x|^{\gamma-Q}u(y)|y|^{\alpha}dy\\&
        =C|x|^{-\alpha}\int_{\frac{|x|}{2}\leq |y|\leq 2|x|}|y^{-1}x|^{\gamma-Q}\tilde{u}(y)dy\\&
        \leq C|x|^{-\alpha}\int_{|y|\leq 2|x|}|y^{-1}x|^{\gamma-Q}\tilde{u}(y)dy\\&
        \stackrel{\eqref{ocenemb1}}\leq C|x|^{-\alpha}\int_{|y^{-1}x|\leq 3(r+|z|)}|y^{-1}x|^{\gamma-Q}\tilde{u}(y)dy\\&
        =C|x|^{-\alpha}\sum_{k=0}^{\infty}\int_{3(r+|z|)2^{-k-1}\leq |y^{-1}x|\leq 3(r+|z|)2^{-k}}|y^{-1}x|^{\gamma-Q}\tilde{u}(y)dy
        \end{split}
\end{equation*}
        \begin{equation}\label{j2}
    \begin{split}
        &\leq C|x|^{-\alpha}\sum_{k=0}^{\infty}\left(3(r+|z|)2^{-k}\right)^{\gamma-Q}\int_{|y^{-1}x|\leq 3(r+|z|)2^{-k}}\tilde{u}(y)dy\\&
        \leq C|x|^{-\alpha}M_{0}\tilde{u}(x)\sum_{k=0}^{\infty}\left(3(r+|z|)2^{-k}\right)^{\gamma}\\&
        \stackrel{\gamma>0}\leq C|x|^{-\alpha}(r+|z|)^{\gamma}M_{0}\tilde{u}(x)\\&
        \stackrel{\eqref{ocenemb2}}\leq C|x|^{-\alpha}(2r+|x|)^{\gamma}M_{0}\tilde{u}(x).
    \end{split}
\end{equation}
By using \eqref{j2}, \eqref{ocenemb3} and the boundedness of the Hardy-Littlewood maximal operator $M_{0}:M^{\lambda}_{p}(\G)\rightarrow M^{\lambda}_{p}(\G)$, we have
\begin{equation*}
    \begin{split}
        \int_{|z^{-1}x|<r}|x|^{-\beta p}J_{2}^{p}(x)dx&\stackrel{\eqref{j2}}\leq C\int_{|z^{-1}x|<r}|x|^{-\gamma p}(2r+|x|)^{\gamma p}|M_{0}\tilde{u}(x)|^{p}dx\\&
        =C\int_{|z^{-1}x|<r}\left(\frac{2r}{|x|}+1\right)^{\gamma p}|M_{0}\tilde{u}(x)|^{p}dx\\&
        \stackrel{\eqref{ocenemb3}}\leq C\int_{|z^{-1}x|<r}\left(\frac{2}{\theta-1}+1\right)^{\gamma p}|M_{0}\tilde{u}(x)|^{p}dx\\&
        \leq C\int_{|z^{-1}x|<r}|M_{0}\tilde{u}(x)|^{p}dx\\&
        \stackrel{\eqref{bounhl}}\leq Cr^{\lambda}\|\tilde{u}\|^{p}_{M^{\lambda}_{p}(\G)}\\&
        =Cr^{\lambda}\||\cdot|^{\alpha}u\|^{p}_{M^{\lambda}_{p}(\G)}.
    \end{split}
\end{equation*}

\textbf{Step 4}. Finally, by using conclusions of Step 1-Step 3, we arrive at
\begin{equation*}
    \int_{|z^{-1}x|<r}|x|^{-\beta q}|I_{\gamma}u(x)|^{q}dx\leq C\int_{|z^{-1}x|<r}|x|^{-\beta q}\left(\sum_{i=1}^{3}J^{q}_{i}(x)\right)dx\leq Cr^{\lambda}\||\cdot|^{\alpha}u\|^{q}_{M^{\lambda}_{p}(\G)},
\end{equation*}
completing the proof.
\end{proof}

\section{Consequences of the Adams type Stein-Weiss inequality}
This section presents some consequences of the Stein-Weiss-Adams inequality on the stratified group and  Euclidean settings.
\subsection{Stratified group setting}
In this subsection, we establish the integer order Hardy, Hardy-Sobolev, Rellich, and Gagliardo-Nirenberg inequalities on stratified groups.
\begin{thm}[Weighted Hardy inequality]\label{HIG}
    Let $\G$ be a stratified group with homogeneous dimension $Q$ and let $|\cdot|$ be a norm. Assume that $1<p<\infty$, $\alpha<\frac{Q}{p'}$, $\beta<\frac{Q-\lambda}{p}$, $\alpha+\beta=1$ and $0<\lambda<\min\{Q,\,\,Q-\beta p\}$. Then for all $|\cdot|^{\alpha}\nabla_{\G} u\in M^{\lambda}_{p}(\G)$ we have
    \begin{equation}\label{eq5-1}
        \left\|\frac{u}{|\cdot|^{\beta}}\right\|_{M^{\lambda}_{p}(\G)}\leq C \||\cdot|^{\alpha}\nabla_{\G} u\|_{M^{\lambda}_{p}(\G)}.
    \end{equation}
\end{thm}
\begin{cor}
    By substituting $\alpha=0$ in \eqref{eq5-1}, we obtain the Hardy inequality
    \begin{equation}\label{CH1}
         \left\|\frac{u}{|\cdot|}\right\|_{M^{\lambda}_{p}(\G)}\leq C \|\nabla_{\G} u\|_{M^{\lambda}_{p}(\G)},\,\,\,0<\lambda<Q-p.
    \end{equation}
\end{cor}
\begin{cor}[Uncertainly principle on Morrey space]\label{UP1}
   Let $\G$ be a stratified group with homogeneous dimension $Q$ and $|\cdot|$ be a quasi-norm.  Then  we have
    \begin{equation}
         \left\|u\right\|_{M^{\lambda}_{2}(\G)}\leq C \||\cdot| u\|_{M^{\lambda}_{2}(\G)}\|\nabla_{\G} u\|_{M^{\lambda}_{2}(\G)},\,\,\,0<\lambda<Q-2.
    \end{equation}
\end{cor}
\begin{proof}[Proof of Corollary \ref{UP1}]
    By combining H\"{o}lder's inequality and \eqref{CH1} with $p=2$, we obtain
    \begin{equation*}
        \begin{split}
          \int_{B(x,r)}|u(y)|^{2}dy&=\int_{B(x,r)}|y||u(y)||y|^{-1}|u(y)|dy\\&
          \leq \left(\int_{B(x,r)}|y|^{2}|u(y)|^{2}dy\right)^{\frac{1}{2}}\left(\int_{B(x,r)}|y|^{-2}|u(y)|^{2}dy\right)^{\frac{1}{2}}\\&
          \leq r^{\lambda} \||\cdot| u\|_{M^{\lambda}_{2}(\G)}\left\|\frac{u}{|\cdot|}\right\|_{M^{\lambda}_{2}(\G)}\\&
          \stackrel{\eqref{CH1}}\leq r^{\lambda} \||\cdot| u\|_{M^{\lambda}_{2}(\G)}\|\nabla_{\G} u\|_{M^{\lambda}_{2}(\G)},
        \end{split}
    \end{equation*}
    completing the proof.
\end{proof}
\begin{proof}[Proof of Theorem \ref{HIG}]
     By applying pointwise estimate from \cite[p. 280]{BLU07}, we have
    \begin{equation*}
        |u(x)|\leq C\int_{\G}\frac{|\nabla_{\G} u(y)|}{|y^{-1}x|^{Q-1}}dy=CI_{Q-1}(|\nabla_{\G} u|).
    \end{equation*}
    By taking $\alpha+\beta=\gamma=1$ in Theorem \ref{stein-weissthm}, we get
    \begin{equation*}
        Q>\gamma=\alpha+\beta=1\,\,\,\Rightarrow\,\,\,q=p,
        \end{equation*}
        \begin{equation*}
        \lambda<Q-\beta p,\,\,\,\alpha<\frac{Q}{p'},\,\,\,0<\lambda<Q,
    \end{equation*}
    and
    \begin{equation*}
        \beta<\frac{Q-\lambda}{p}\,\,\,\Rightarrow\,\,\,\lambda<Q-\beta p.
    \end{equation*}
    By combining last facts with \eqref{stein-weiss}, we get
    \begin{equation*}
        \begin{split}
            \left\|\frac{u}{|\cdot|^{\beta}}\right\|_{M^{\lambda}_{p}(\G)}\leq C\left\|\frac{I_{Q-1}(|\nabla_{\G} u|)}{|\cdot|^{\beta}}\right\|_{M^{\lambda}_{p}(\G)}\stackrel{\eqref{stein-weiss}}\leq C\||\cdot|^{\alpha}\nabla u\|_{M^{\lambda}_{p}(\G)},
        \end{split}
    \end{equation*}
    completing the proof.
\end{proof}
Here we state the Hardy-Sobolev inequality in Morrey spaces.
\begin{thm}[Weighted Hardy-Sobolev inequality]\label{HSG}
    Let $\G$ be a stratified group with homogeneous dimension $Q$ and $|\cdot|$ be a norm.
    Let  $\alpha,\beta\in \mathbb{R}$, $0\leq \alpha+\beta\leq 1 <Q$, $1<p<\frac{Q}{1-\alpha-\beta}$, $\frac{1}{q}=\frac{1}{p}-\frac{1-\alpha-\beta}{Q-\lambda}$,  $\alpha<\frac{Q}{p'}$, $\beta<\frac{Q-\lambda}{q}$ and $0<\lambda<\min\{Q-\beta p, Q-(1-\alpha-\beta)p\}$. Then, for all  $|\cdot|^{\alpha}u\in M^{\lambda}_{p}(\G)$, we have
\begin{equation}\label{CHS}
\left\|\frac{u}{|\cdot|^{\beta}}\right\|_{M^{\lambda}_{q}(\G)}\leq C \||\cdot|^{\alpha}\nabla_{\G} u\|_{M^{\lambda}_{p}(\G)},
\end{equation}
where $C$ is a positive constant independent of $u$.
\end{thm}
\begin{cor}
    By taking $\alpha=\beta=0$ and $\beta=\gamma=1$ with $\alpha=0$ in the previous theorem, we get the Sobolev inequality and the Hardy inequality, respectively.
\end{cor}
\begin{proof}[Proof of Theorem \ref{HSG}]
    The proof of this theorem is similar to the one of Theorem \ref{HIG}.
\end{proof}
Also, we present the Morrey space version of the Rellich inequality.
\begin{thm}[Weighted Rellich inequality]\label{RG}
    Let $\G$ be a stratified group with homogeneous dimension $Q$ and $|\cdot|$ be a norm.
    Let  $\alpha,\beta\in \mathbb{R}$, $\alpha+\beta=2$, $1<p<\infty$,  $\alpha<\frac{Q}{p'}$, $\beta<\frac{Q-\lambda}{p}$ and $0<\lambda<\min\{Q,Q-\beta p\}$. Then, for all  $|\cdot|^{\alpha}u\in M^{\lambda}_{p}(\G)$, we have
\begin{equation}\label{RGI}
\left\|\frac{u}{|\cdot|^{\beta}}\right\|_{M^{\lambda}_{p}(\G)}\leq C \||\cdot|^{\alpha}\Delta_{\G} u\|_{M^{\lambda}_{p}(\G)},
\end{equation}
where $C$ is a positive constant independent of $u$.
\end{thm}
\begin{cor}
    By taking $\alpha=0$ and $p=2$, we get the classical Rellich inequality in the following form:

        \begin{equation}\label{RGIs}
\left\|\frac{u}{|\cdot|^{2}}\right\|_{M^{\lambda}_{2}(\G)}\leq C \|\Delta_{\G} u\|_{M^{\lambda}_{2}(\G)},\,\,\,\,\,0<\lambda<Q-4.
\end{equation}
\end{cor}
\begin{proof}[Proof of Theorem \ref{RG}]
    By \cite[Theorem 5.3.3]{BLU07}, we have that the
    \begin{equation}
        u(x)=-\int_{\G}\Gamma(y^{-1}x)\Delta_{\G}u(y)dy,
    \end{equation}
    where $\Gamma(\cdot)$ is the fundamental solution of the sub-Laplacian and also, by \cite[Theorem 5.5.6]{BLU07} we have $\Gamma(\cdot)=C_{Q}|\cdot|^{2-Q}$, where $C_{Q}>0$. Hence, we have that the
    \begin{equation}
        |u(x)|\leq \int_{\G}|\Gamma(y^{-1}x)||\Delta_{\G}u(y)|dy\leq C\int_{\G}|y^{-1}x|^{2-Q}|\Delta_{\G}u(y)|dy=CI_{Q-2}(|\Delta_{\G}u|).
    \end{equation}
     By taking $\alpha+\beta=\gamma=2$ in Theorem \ref{stein-weissthm}, we get
     \begin{equation*}
        \begin{split}
            \left\|\frac{u}{|\cdot|^{\beta}}\right\|_{M^{\lambda}_{p}(\G)}\leq C\left\|\frac{I_{Q-2}(|\Delta_{\G} u|)}{|\cdot|^{\beta}}\right\|_{M^{\lambda}_{p}(\G)}\stackrel{\eqref{stein-weiss}}\leq C\||\cdot|^{\alpha}\Delta u\|_{M^{\lambda}_{p}(\G)},
        \end{split}
    \end{equation*}
    completing the proof.
\end{proof}
We also derive the Gagliardo-Nirenberg inequality.
\begin{thm}[Gagliardo-Nirenberg inequality]\label{GNG}
Let $\G$ be a stratified group with homogeneous dimension $Q$ and $|\cdot|$ be a norm.
  Let  $1<p<Q$ and $0<\lambda<Q-p$. Assume that $\frac{1}{q}=a\left(\frac{1}{p}-\frac{1}{Q-\lambda}\right)+
\frac{1-a}{r}$, $a\in[0,1]$, $q>1$ and $r\geq1$, then we have
\begin{equation}\label{GNi}
\left\|u\right\|_{M^{\lambda}_{q}(\G)}\leq C\|\nabla_{\G}u\|^{a}_{M^{\lambda}_{p}(\G)}\|u\|^{1-a}_{M^{\lambda}_{r}(\G)},
\end{equation}
where $C$ is a positive constant independent of $u$.
\end{thm}
\begin{proof}
By using the H\"{o}lder inequality, for every $\frac{1}{q}=a\left(\frac{1}{p}-\frac{1}{Q-\lambda}\right)+
\frac{a}{r}$, we obtain
\begin{equation*}
\begin{split}
    \int_{B(x,r)}|u(x)|^{q}dx&=\int_{B(x,r)}|u(x)|^{aq}|u(x)|^{(1-a)q}dx\\&
    \leq\left(\int_{B(x,r)}|u(x)|^{\frac{(Q-\lambda)p}{Q-\lambda-p}}dx\right)^{aq\left(\frac{1}{p}-\frac{1}{Q-\lambda}\right)}\left(\int_{B(x,r)}|u(x)|^{r}dx\right)^{\frac{(1-a)q}{r}}\\&
    =\left(\frac{r^{\lambda}}{r^{\lambda}}\int_{B(x,r)}|u(x)|^{\frac{(Q-\lambda)p}{Q-\lambda-p}}dx\right)^{aq\left(\frac{1}{p}-\frac{1}{Q-\lambda}\right)}\left(\frac{r^{\lambda}}{r^{\lambda}}\int_{B(x,r)}|u(x)|^{r}dx\right)^{\frac{(1-a)q}{r}}\\&
    \leq r^{\lambda aq\left(\frac{1}{p}-\frac{1}{Q-\lambda}\right)+\lambda\frac{(1-a)q}{r}}\|u\|^{aq}_{M^{\lambda}_{\frac{(Q-\lambda)p}{Q-\lambda-p}}(\G)}\|u\|^{(1-a)q}_{M^{\lambda}_{r}(\G)}\\&
    \stackrel{\eqref{CHS},\alpha=\beta=0}\leq  Cr^{\lambda }\|\nabla_{\G}u\|^{aq}_{M^{\lambda}_{p}(\G)}\|u\|^{(1-a)q}_{M^{\lambda}_{r}(\G)},
\end{split}
\end{equation*}
    completing the proof.
\end{proof}
\subsection{Euclidean setting}
This subsection presents some consequences of the Stein-Weiss-Adams inequality in the Euclidean setting. Many of these inequalities are new already on $\R$.

Firstly, let us state the integer and fractional Hardy inequalities for  Morrey spaces.

\begin{thm}[Integer and fractional Hardy inequalities]\label{Hardy}
Assume that $1<p<\infty$, $\alpha<\frac{N}{p'}$, $\beta<\frac{N-\lambda}{p}$, $\alpha+\beta=1$ and $0<\lambda<\min\{N,\,\,N-\beta p\}$. Then for all $|\cdot|_{E}^{\alpha}\nabla u\in M^{\lambda}_{p}(\R)$ we have the integer Hardy inequality
 \begin{equation}\label{ehar}
        \left\|\frac{u}{|x|^{\beta}_{E}}\right\|_{M^{\lambda}_{p}(\mathbb{R}^{N})}\leq C\||\cdot|_{E}^{\alpha}\nabla u\|_{M^{\lambda}_{p}(\mathbb{R}^{N})}.
    \end{equation}
Moreover, when $\alpha=0$, we have the classical Hardy inequality on Morrey space
 \begin{equation}\label{euhar}
        \left\|\frac{u}{|x|_{E}}\right\|_{M^{\lambda}_{p}(\mathbb{R}^{N})}\leq C\|\nabla u\|_{M^{\lambda}_{p}(\mathbb{R}^{N})},\,\,\,\,\,N>p,\,\,0<\lambda<N-p.
    \end{equation}

Also, if  $1<p<\infty$,  $\alpha<\frac{N}{p'}$, $\beta<\frac{N-\lambda}{p}$, $\alpha+\beta=\gamma\in(0,1)$ and $0<\lambda<\min\{N,\,\,N-\beta p\}$, then, for all  $|\cdot|^{\alpha}_{E}(-\Delta )^{\frac{\gamma}{2}}u\in M^{\lambda}_{p}(\mathbb{R}^{N})$, we have the weighted fractional Hardy inequality
\begin{equation}\label{efhar}
\left\||\cdot|_{E}^{-\beta}u\right\|_{M^{\lambda}_{p}(\mathbb{R}^{N})}\leq C\||\cdot|^{\alpha}_{E}(-\Delta )^{\frac{\gamma}{2}}u\|_{M^{\lambda}_{p}(\mathbb{R}^{N})},
\end{equation}
where $C$ is a positive constant independent of $u$.
In addition, if $\gamma\in(0,1)$, $N>\gamma p$ and $0<\lambda<N-\gamma p$, we have the classical fractional Hardy inequality on the Morrey space
\begin{equation}\label{eufhar}
\left\||\cdot|_{E}^{-\gamma}u\right\|_{M^{\lambda}_{p}(\mathbb{R}^{N})}\leq C\|(-\Delta )^{\frac{\gamma}{2}}u\|_{M^{\lambda}_{p}(\mathbb{R}^{N})}.
\end{equation}
\end{thm}
\begin{rem}
    The unweighted fractional Hardy inequality \eqref{eufhar} was also proved in \cite{GHNS}.
\end{rem}
\begin{cor}[Uncertainly principle on Morrey space]
    If $N>2$ and $0<\lambda<N-2$,  then  we have
    \begin{equation}
         \left\|u\right\|_{M^{\lambda}_{2}(\R)}\leq C \||\cdot| u\|_{M^{\lambda}_{2}(\R)}\|\nabla u\|_{M^{\lambda}_{2}(\R)}.
    \end{equation}
    If $\gamma\in (0,1)$, $N>2\gamma$ and $0<\lambda<N-2\gamma$, then we have
    \begin{equation}
        \left\|u\right\|_{M^{\lambda}_{2}(\R)}\leq C \||\cdot|^{\gamma} u\|_{M^{\lambda}_{2}(\R)}\|(-\Delta )^{\frac{\gamma}{2}} u\|_{M^{\lambda}_{2}(\R)}.
    \end{equation}
\end{cor}
\begin{proof}
 By applying a pointwise estimate from \cite[Lemma 6.26]{CF12} and \cite[Formula (18), p.125]{Sbook}, we have
    \begin{equation*}
        |u(x)|\leq C\int_{\mathbb{R}^{N}}\frac{|\nabla u(y)|}{|x-y|^{N-1}_{E}}dy.
    \end{equation*}
    By combining this with \eqref{stein-weiss}, we get
    \begin{equation*}
        \begin{split}
            \left\|\frac{u}{|x|^{\beta}_{E}}\right\|_{M^{\lambda}_{p}(\mathbb{R}^{N})}\leq C\left\|\frac{I_{1}(|\nabla u|)}{|x|^{\beta}_{E}}\right\|_{M^{\lambda}_{p}(\mathbb{R}^{N})}\stackrel{\eqref{stein-weiss}}\leq C\||\cdot|^{\alpha}\nabla u\|_{M^{\lambda}_{p}(\mathbb{R}^{N})}.
        \end{split}
    \end{equation*}

To prove the fractional version of the Hardy inequality, by combining the fact $$ I_{\gamma} \left[(-\Delta)^{\frac{\gamma}{2}}f\right](x)=f(x),$$ and Remark \ref{remeuc} of Theorem \ref{stein-weissthm}, we obtain \eqref{efhar}.
\end{proof}
Now for (global) Morrey spaces we state the fractional Hardy-Sobolev inequality on $\mathbb{R}^N$.
\begin{thm}[Integer and fractional Hardy-Sobolev inequalities]\label{Hardy-Sob}
 Let  $\alpha,\beta\in \mathbb{R}$, $0\leq \alpha+\beta\leq 1 <N$, $1<p<\frac{N}{1-\alpha-\beta}$, $\frac{1}{q}=\frac{1}{p}-\frac{1-\alpha-\beta}{N-\lambda}$,  $\alpha<\frac{N}{p'}$, $\beta<\frac{N-\lambda}{q}$ and $0<\lambda<\min\{N-\beta p, N-(1-\alpha-\beta)p\}$. Then, for all  $|\cdot|^{\alpha}u\in M^{\lambda}_{p}(\R)$, we have
\begin{equation}\label{CHSE}
\left\|\frac{u}{|\cdot|_{E}^{\beta}}\right\|_{M^{\lambda}_{q}(\R)}\leq C \||\cdot|_{E}^{\alpha}\nabla u\|_{M^{\lambda}_{p}(\R)},
\end{equation}
where $C$ is a positive constant independent of $u$.

Moreover, when $\alpha=0$, we have the classical Hardy-Sobolev inequality

\begin{equation}\label{CHSE1}
\left\|\frac{u}{|\cdot|_{E}^{\beta}}\right\|_{M^{\lambda}_{q}(\R)}\leq C \|\nabla u\|_{M^{\lambda}_{p}(\R)}.
\end{equation}

Let $\gamma\in(0,1)$, $0\leq\alpha+\beta\leq\gamma<N$ and $1<p<\frac{N}{\gamma-\alpha-\beta}$ such that  and $0<\lambda<\min\{N-\beta p, N-(\gamma-\alpha-\beta)p\}$. If $\frac{1}{q}=\frac{1}{p}-\frac{\gamma-\alpha-\beta}{N-\lambda}$, then, we have
\begin{equation}\label{hardysobin}
\left\||\cdot|_{E}^{-\beta}u\right\|_{M^{\lambda}_{q}(\mathbb{R}^{N})}\leq C\||\cdot|^{\alpha}_{E}(-\Delta )^{\frac{\gamma}{2}}u\|_{M^{\lambda}_{p}(\mathbb{R}^{N})},
\end{equation}
where $C$ is a positive constant independent of $f$.

Moreover, when $\alpha=0$, that is, if $0\leq\beta\leq\gamma<N$ and $1<p<\frac{N}{\gamma-\beta}$, $0<\lambda<\min\{N-\beta p, N-(\gamma-\beta)p\}$ such that $\frac{1}{q}=\frac{1}{p}-\frac{\gamma-\beta}{N-\lambda}$, we have the fractional Hardy-Sobolev inequality
\begin{equation}\label{hardysobina}
\left\||\cdot|_{E}^{-\beta}u\right\|_{M^{\lambda}_{q}(\mathbb{R}^{N})}\leq C\|(-\Delta )^{\frac{\gamma}{2}}u\|_{M^{\lambda}_{p}(\mathbb{R}^{N})}.
\end{equation}
\end{thm}
\begin{rem}
In \eqref{hardysobina}, we have the fractional Hardy inequality if $\beta=\gamma$ and  we get the fractional Sobolev inequality if $\beta=0$. Also, 
when $\lambda=0$, we obtain the standard (classical) fractional Hardy-Sobolev inequality.
\end{rem}
\begin{proof}
The proof of this theorem is similar to the one in the previous section.
\end{proof}
Similarly, we present the fractional Rellich inequality.
\begin{thm}[Integer and fractional Rellich inequalities]\label{Rell}
    Let  $1<p<\infty$, $\alpha,\beta\in \mathbb{R}$, $\alpha+\beta=2$,  $\alpha<\frac{N}{p'}$, $\beta<\frac{N-\lambda}{p}$ and $0<\lambda<\min\{N,N-\beta p\}$. Then  we have
\begin{equation}\label{RI}
\left\|\frac{u}{|\cdot|_{E}^{\beta}}\right\|_{M^{\lambda}_{p}(\R)}\leq C \||\cdot|_{E}^{\alpha}\Delta u\|_{M^{\lambda}_{p}(\R)}.
\end{equation}
Moreover, with $\alpha=0$, we establish
\begin{equation}\label{RIr}
\left\|\frac{u}{|\cdot|_{E}^{2}}\right\|_{M^{\lambda}_{p}(\R)}\leq C \|\Delta u\|_{M^{\lambda}_{p}(\R)},\,\,\,p>1,\,\, N>2p,\,\, 0<\lambda<N-2p.
\end{equation}

Also, let   $p>1$, $\alpha<\frac{N}{p'}$, $\beta<\frac{N-\lambda}{p}$, $\alpha+\beta=\gamma\in(1,2)$, $N>\gamma p$ and   $0<\lambda<\min\{N,N-\gamma p\}$. Then,  we have
\begin{equation}\label{wrellin}
\left\||\cdot|_{E}^{-\beta}u\right\|_{M^{\lambda}_{p}(\mathbb{R}^{N})}\leq C\||\cdot|_{E}^{-\alpha}(-\Delta )^{\frac{\gamma}{2}}u\|_{M^{\lambda}_{p}(\mathbb{R}^{N})}.
\end{equation}
In addition, if $\alpha=0$, then we have
\begin{equation}\label{urellin}
\left\||\cdot|_{E}^{-\gamma}u\right\|_{M^{\lambda}_{p}(\mathbb{R}^{N})}\leq C\|(-\Delta )^{\frac{\gamma}{2}}u\|_{M^{\lambda}_{p}(\mathbb{R}^{N})},\,\,\,\,\gamma\in(1,2),\,\,p>1,\,\,N>\gamma p,\,\,0<\lambda< N-\gamma p.
\end{equation}
\end{thm}

\begin{rem}
With  $p=2$, we get the classical Rellich inequality on the Morrey space
\begin{equation}
 \left\||\cdot|^{-2}f\right\|_{M^{\lambda}_{2}(\mathbb{R}^{N})}\leq C\|\Delta f\|_{M^{\lambda}_{2}(\mathbb{R}^{N})},\,\,\,\,\text{for}\,\,0<\lambda<N-4,
\end{equation}
and
\begin{equation}
 \left\||\cdot|^{-\gamma}u\right\|_{M^{\lambda}_{2}(\mathbb{R}^{N})}\leq C\|(-\Delta )^{\frac{\gamma}{2}}u\|_{M^{\lambda}_{2}(\mathbb{R}^{N})},\,\,\,\,\text{for}\,\,\,\gamma\in(1,2),\,\,0<\lambda<N-2\gamma.
\end{equation}
This is a Morrey space extension of the classical Rellich inequality (with $\lambda=0$):
\begin{equation}
 \left\||\cdot|^{-2}u\right\|_{L^{2}(\mathbb{R}^{N})}\leq C\|\Delta u\|_{L^{2}(\mathbb{R}^{N})},\,\,\,\,N\geq5,
\end{equation}
and
\begin{equation}
 \left\||\cdot|^{-\gamma}u\right\|_{L^{2}(\mathbb{R}^{N})}\leq C\|(-\Delta )^{\frac{\gamma}{2}}u\|_{L^{2}(\mathbb{R}^{N})},\,\,\,\,\text{for}\,\,N>2\gamma,\,\,\,\gamma\in(1,2).
\end{equation}
\end{rem}
Finally, we present the fractional Gagliardo-Nirenberg inequality for global Morrey spaces.
\begin{thm}[Integer and fractional Gagliardo-Nirenberg inequalities]
Let  $1<p<N$ and $0<\lambda<N-p$. Assume that $\frac{1}{q}=a\left(\frac{1}{p}-\frac{1}{N-\lambda}\right)+
\frac{1-a}{r}$, $a\in[0,1]$, $q>1$ and $r\geq1$, then we have
\begin{equation}\label{GNir}
\left\|u\right\|_{M^{\lambda}_{q}(\R)}\leq C\|\nabla u\|^{a}_{M^{\lambda}_{p}(\G)}\|u\|^{1-a}_{M^{\lambda}_{r}(\R)},
\end{equation}
where $C$ is a positive constant independent of $u$.

Let $\gamma\in(0,1),$ $1<p<\frac{N}{\gamma}$ and $0<\lambda<N-\gamma p$. Assume that $\frac{1}{q}=a\left(\frac{1}{p}-\frac{\gamma}{N-\lambda}\right)+
\frac{1-a}{r}$, $a\in[0,1]$, $q>1$ and $r\geq1$, then we have
\begin{equation}\label{GN}
\left\|u\right\|_{M^{\lambda}_{q}(\G)}\leq C\|(-\Delta )^{\frac{\gamma}{2}}u\|^{a}_{M^{\lambda}_{p}(\mathbb{R}^{N})}\|u\|^{1-a}_{M^{\lambda}_{r}(\mathbb{R}^{N})},
\end{equation}
where $C$ is a positive constant independent of $u$.
\end{thm}

\section{Acknowledgments}
The second author would like to thank the Faculty of Fundamental Science, Industrial University of Ho Chi Minh City, Vietnam, for the opportunity to work in it.
This research was funded by the Committee of Science of the Ministry of Science and Higher Education of the Republic of Kazakhstan (Grant No. AP19674900) and
Nazarbayev University Program 20122022CRP1601. The authors were also supported in part by the FWO Odysseus Project G.0H94.18N, the Methusalem programme of the Ghent University Special Research Fund (BOF) (Grant number 01M01021), and EPSRC Grant EP/R003025.

\end{document}